\documentclass[11pt]{article}
\usepackage{graphicx} 
\usepackage{amssymb,amsthm,amsmath,mathrsfs,fullpage, hyperref}
\usepackage{tikz}
\usetikzlibrary{matrix}
\usepackage{makecell}
\usepackage{diagbox}
\usepackage{enumitem}

\DeclareMathOperator{\Sav}{Sav}

\newcommand\A{\mathcal A}
\renewcommand\S{\mathcal S}
\newcommand\e{\varepsilon}

\newtheorem{theorem}{Theorem}[section]
\newtheorem{lemma}[theorem]{Lemma}
\newtheorem{proposition}[theorem]{Proposition}
\newtheorem{corollary}[theorem]{Corollary}

\theoremstyle{definition}

  \newtheorem{question}[theorem]{Question}
\newtheorem{remark}[theorem]{Remark}
\newtheorem{example}[theorem]{Example}

\newcommand{\todo}[1]{\vspace{2 mm}\par\noindent
   \marginpar{\textsc{ToDo}}\framebox{\begin{minipage}[c]{0.95 \textwidth}
   \tt #1 \end{minipage}}\vspace{2 mm}\par}

\usepackage{authblk}

\title{Permutations that strongly avoid 132}
\date{}
\author[1]{Kassie Archer}
\author[2]{Christina Graves}
\affil[1]{{\small Department of Mathematics, United States Naval Academy, Annapolis, MD, 21402}}
\affil[2]{{\small Department of Mathematics, University of Texas at Tyler, Tyler, TX, 75799}}

\affil[ ]{{\small Email: karcher@usna.edu, cgraves@uttyler.edu}}

\begin{document}

\maketitle

\begin{abstract} A permutation $\pi$ strongly avoids the pattern $\tau$ if both $\pi$ and $\pi^2$ avoid $\tau$. In this paper, we enumerate permutations of size $n$ that strongly avoid the pattern 132. This enumeration allows us to prove a conjecture that the growth rate of such permutations is 2.
\end{abstract}


\section{Introduction}\label{sec:intro}

In \cite{BS2019}, B\'{o}na and Smith introduce the notion of strong pattern avoidance, in which a permutation $\pi$ and its square $\pi^2$ in the symmetric group (both written in one-line notation) both avoid a given pattern $\tau$. In that paper, they show that the number of permutations that strongly avoid the monotone increasing pattern is eventually zero and that the number that avoid 312 (equivalently, 231) is given by the coefficients of the rational generating function 
\[
\Sav_{312}(x) = \frac{1-x-x^2+x^3}{1-2x-x^2+2x^3-x^4}.
\]
They leave open the question of enumerating permutations that strongly avoid 321 or 132 (equivalently 213), though they do provide some bounds. 

In the case of strong 321 avoidance, they bound the sequence below by the number of so-called block cyclic permutations, which have a growth rate of approximately 2.3247. In the case of strong 132 avoidance, they bound the sequence below by 132-avoiding involutions, which have the property that their square is the identity and thus necessarily avoids 132. They state the following question(s). Here, $\Sav_n(132)$ denotes the set of strongly 132-avoiding permutations of size $n$.

\begin{question}[{\cite[Questions 5.1 \& 5.2]{BS2019}}]\label{ques:bona} 
    Is the growth rate of $|\Sav_n(132)|$ equal to 2? Is it true that $|\Sav_n(132)| \leq 2^n$ for all $n$?
\end{question}

The study of strong avoidance continued in \cite{BD2020}, where Burcroff and Defant enumerate permutations strongly avoiding certain pairs of patterns. They also introduce the notion of powerful pattern avoidance in which $\pi^i$ avoids a pattern $\tau$ for all integers $i\geq 1.$ 
In \cite{AG2024}, a variation of strong avoidance called chain avoidance was introduced. In that case, a permutation $\pi$ avoids $\sigma$, while $\pi^2$ avoids $\tau.$ Some other recent results regarding strong and chain avoidance can be found in \cite{BD2020,P23,PG2025}.

As is stated in \cite{BS2019, BD2020}, there appears to be a relationship between permutations that strongly avoid a pattern $\tau$ and the cycle structure of $\tau$-avoiding permutations. In particular, both papers mention the question of enumerating 132-avoiding permutations composed only of $3$-cycles and fixed points, which was later answered in \cite{AG2022, AL2024}. As we will see in this paper, the results from \cite{AG2022} are indeed essential to answering the primary question here.

In this paper, we give the exact enumeration of those permutations that strongly avoid the pattern 132. In the process, we answer Question~\ref{ques:bona}. 
Our primary strategy for approaching this question is to consider those strongly 132-avoiding permutations for which $n$ is in a $k$-cycle. In Section~\ref{sec:lessthan3}, we consider the cases in which $k\leq 3.$  These results rely heavily on previous work done in the field, particularly work by Simion and Schmidt \cite{SS1985} on pattern-avoiding involutions and the authors of this paper \cite{AG2022} on pattern-avoiding permutations composed only of 3-cycles. 
The remainder of the paper focuses on the cases where $k\geq 4$. In these cases, we find that $k$ must divide $n$.

We enumerate permutations of size $n$ that strongly avoid the pattern 132 with the following generating function.
\begin{theorem}\label{thm:main} The number of permutations that strongly avoid the pattern 132 is given by coefficients of the generating function
\[
\Sav_{132}(x)
=\frac{
c\bigl(x^3 c(x^3)\bigr)}{{
2-(1+x)c\bigl(x^3 c(x^3)\bigr)}}\Bigg(\dfrac{1-x}{1-x-x^2 c(x^2)}
\;+\;\dfrac{2x(1+2x)\bigl(c(x^3)-1\bigr)}{1-x^2}\Bigg),
\]
where $c(x)$ is the generating function for the Catalan numbers.
\end{theorem}

As a corollary, we answer Question \ref{ques:bona} in the affirmative, finding that the growth rate of $|\Sav_n(132)|$ is indeed 2 and that $|\Sav_n(132)|<2^n$ for large enough $n.$

\subsection{Definitions and notation}
Let us first give a few standard definitions and notation. Let $\S_n$ denote the set of permutations of $[n]=\{1,2,\ldots, n\}$. The standard one-line notation of a permutation $\pi \in \S_n$ is given by $\pi = \pi_1\pi_2\cdots\pi_n$ where $\pi_i := \pi(i)$. To avoid complicated subscripts, in this paper we use the notation $\pi_i$ and $\pi(i)$ interchangeably. A permutation $\pi = \pi_1\pi_2\cdots\pi_n$ \emph{avoids} a pattern $\tau$ if there does not exist a subsequence of elements of $\pi$ in the same relative order as $\tau$. We denote the set of permutations in $\S_n$ that avoid a pattern $\tau$ by $\S_n(\tau)$. Furthermore, a permutation $\pi$ \emph{strongly avoids} $\tau$ if both $\pi$ and $\pi^2$ avoid $\tau.$ 

For integers $n_1$ and $n_2$, let $[n_1,n_2]$ denote the set of consecutive integers $\{n_1, n_1+1, \ldots, n_2\}$ when $n_1 \leq n_2$ and the empty set otherwise. Furthermore, if $n_1 \leq n_2$, let $\pi([n_1, n_2]) = \{\pi_{n_1}, \pi_{n_1+1}, \ldots, \pi_{n_2}\}$. The \emph{inverse} of a permutation $\pi$, denoted $\pi^{-1}$, is defined by $\pi_i^{-1}=j$ where $\pi_j=i$.

Let us use $\A_n$ to denote the set of permutations of size $n$ that strongly avoid $132$ (i.e., $\A_n:=\Sav_{132}(n)$) and let $a_n=|\A_n|$. Given a permutation in $\A_n$, we consider its cycle notation and the length of the cycle containing $n$. We will show that the length of the cycle containing $n$ is either 1, 2, 3, or some $k$ dividing $n$. We thus enumerate $\A_n$ by the length of the cycle containing $n$. To this end, let $\A_{n,k}$ denote the set of permutations of size $n$ that strongly avoid $132$ and where $n$ is in a cycle of length $k$. A table providing values of $a_{n,k} = |\A_{n,k}|$ for small $n$ and $k$ is shown in Table~\ref{fig:ank}.

\begin{table}
\begin{center}
\begin{tabular}{l||rrrrrrrrrrrrrrr}
\diagbox{$k$}{$n$}&
\ 1 & 2 & 3 & 4 & 5 & 6 & 7 & 8 & 9 & 10 & 11 & 12 & 13 & 14 & 15\\
\hline \hline
1 & 1 & 1 & 2 & 5 & 12 & 24 & 50 & 101 & 202 & 398 & 806 & 1568 & 3148 & 6198 & 12306 \\
2 & & 1 & 1 & 3 & 4 & 10 & 15 & 35 & 56 & 126 & 210 & 462 & 792 & 1716 & 3033\\
3 & & & 2 & 2 & 4 & 14 & 28 & 56 & 132 & 262 & 524 & 1098 & 2202 & 4316 & 8858\\
4 & & & & 2 & - & - & - & 4 & - & - & - & 10 &-&- & - \\
5 & & & & & 4 & - & - & - & - & 8 & - & - & -&- & 20\\
6 & & & & & & 2 & - & - & - & - & - & 4 & -&- & -\\
7 & & & & & & & 8 & - & - & - & - & - & -&36 & -\\
8 & & & & & & & & 6 & - & - & - & - & -&- & -\\
9 & & & & & & & & & 8 & - & - & - & -&- & -\\
10 & & && & & & & & & 12 & - & - & -&- & -\\
11 & & && & & & & & & & 28 & - & -&- & -\\
12 & & && & & & & & & & & 6 & -&- & -\\
13 & & && & & & & & & & & & 56&- & -\\
14 & & && & & & & & & & & & &40 & - \\
15 & & && & & & & & & & & & & & 36 \\
\hline \hline
Total & 1 & 2 & 5 & 12 & 24 & 50 & 101 & 202 & 398 & 806 & 1568 & 3148 & 6198 & 12306 & 24223
\end{tabular}
\caption{Number of permutations of size $n$ strongly avoiding 132 with $n$ in a cycle of length $k$. This number is denoted by $a_{n,k}$. Note that for $k \geq 4$, $a_{n,k}>0$ only when $k$ divides $n$.}
\label{fig:ank}
\end{center}
\end{table}

\section{$\A_{n,k}$ for $k \leq 3$} \label{sec:lessthan3}

In this section, we enumerate those strongly 132-avoiding permutations where $n$ is in a cycle of size at most 3.

\subsection{$\A_{n,1}$ and $\A_{n,2}$}

Let us first count all permutations of size $n$ that strongly avoid 132 where $n$ is in a 1-cycle, or equivalently, $n$ is in the last position. In this case, it is clear that deleting $n$ results in a strongly avoiding 132 permutation of size $n-1$. Conversely, given a permutation of size $n-1$ that strongly avoids 132, we can add $n$ to the end of the permutation to create one with $n$ in a 1-cycle. We state this in the following proposition.

\begin{proposition}
Let $a_{n,1}$ be the number of permutations in $\A_n$ where $n$ is in a 1-cycle. Then, $a_{1,1} = 1$, and for $n \geq 2,$
\[ a_{n,1} = a_{n-1}.\]
Equivalently, the generating function for $a_{n,1}$ is given by
\[ a_1(x) = xa(x),\]
where $a(x)$ is the generating function for $a_n$.
\end{proposition}

Now consider the case where  $n$ is in a 2-cycle. 
We will first prove that if $n$ is in a 2-cycle, then every other element is either in a 2-cycle or is a fixed point. 

\begin{lemma}\label{lem:involution}
    For $n\geq 2$, if $\pi\in\A_n$ and $n$ is in a 2-cycle, then $\pi$ is an involution. 
\end{lemma}

\begin{proof}
Let $\pi \in \A_n$ with $\pi_r=n$ and $\pi_n = r$. Since $\pi$ is 132-avoiding, and $n$ is in position $r$, the entries before $n$ must be larger than the entries after $n$. Thus, 
\[\pi([1,r-1]) = [n-r+1,n-1],\]
and
\[ \pi([r+1,n]) = [1, n-r]. \]
Since $\pi_n=r,$ such a permutation is only possible if $r \leq n-r$, or $r \leq n/2.$ So we assume for the remainder of the proof that $r \leq n/2$.

We must also have
\[ \pi([n-r+1,n-1]) = [1,r-1]\]
because if $\pi_j > r$ for $j \in [n-r+1, n-1]$, then $\pi_jr$ would be the 32 in some 132 pattern.

Now consider $\pi^2.$ Based on $\pi$, it must be that
\[ \pi^2([1,r-1])= [1,r-1],\]
and
\[ \pi^2([n-r+1,n-1]) = [n-r+1, n-1],\] and therefore
\[ \pi^2([r+1,n-r])= [r+1, n-r].\]
Since $r \leq \lfloor n/2 \rfloor$, all elements in $[n-r+1,n-1]$ are greater than $r$, and thus all elements appearing after $r$ in $\pi^2$ are larger than $r$. These elements must appear in increasing order to avoid a 132 pattern starting with $r$. Thus
\[ \pi^2 = \pi^2_1\pi^2_2\cdots\pi_{r-1}^2\  r\  (r+1) \cdots  n.\]
Additionally, since $\pi([n-r+1,n-1]) = [1,r-1]$ and $\pi^2_j=j$ for all $n-r+1\leq j \leq n-1$, these elements are all in transpositions and so $\pi^2_j=j$ for $1\leq j \leq r-1$ as well. Since we have shown that $\pi^2$ is the identity permutation, the proof is complete. 
\end{proof}

The number of involutions avoiding the pattern 132 was given in 1985 by Simion and Schmidt \cite{SS1985}, as stated in the following lemma. 
\begin{lemma}[{\cite[Prop. 5]{SS1985}}]
    \label{lem:number-involutions}
    For $n\geq 1$ the number of involutions avoiding the pattern 132 is given by 
    \[
    \binom{n}{\lfloor n/2\rfloor}.
    \]
\end{lemma}

Using these two lemmas, we can obtain the following result. 

    \begin{proposition} Let $n \geq 2$, and let $a_{n,2}$ be the number of permutations in $\A_n$ where $n$ is in a transposition. Then
\[ a_{n,2} = {n-1 \choose \lfloor \frac{n-2}{2} \rfloor}.\]
Equivalently, the generating function for $a_{n,2}$ is given by 
\[ a_2(x) = \frac{x^{2} c(x^{2})}{1 - x - x^{2} c(x^{2})}\]
where $c(x)$ is the generating function for the Catalan numbers.
\end{proposition}

\begin{proof}
    By Lemma~\ref{lem:involution}, $a_{n,2}$ is exactly the number of involutions of length $n$ that avoid 132 where $n$ is not a fixed point. By Lemma~\ref{lem:number-involutions}, there are $\binom{n}{\lfloor n/2\rfloor}$ involutions total, and thus $\binom{n-1}{\lfloor (n-1)/2\rfloor}$ that have $n$ as a fixed point. Thus the total number of strongly 132-avoiding permutations with $n$ in a 2-cycle is
    \[
    \binom{n}{\lfloor n/2\rfloor}-\binom{n-1}{\lfloor (n-1)/2\rfloor} = {n-1 \choose \lfloor \frac{n-2}{2} \rfloor}.
    \]

    It is well-known that the generating function for $\binom{n}{\lfloor n/2\rfloor}$ is \[d(x) = \frac{1}{1 - x - x^2c(x^2)}\] where $c(x)$ is the generating function for the Catalan numbers
    (see for example, OEIS \cite[A001405]{OEIS}). Thus, $a_2(x) = d(x)-xd(x)-1,$ which can be simplified to the generating function given in the statement of the theorem.
\end{proof}

\subsection{$\A_{n,3}$}

We now consider permutations that strongly avoid 132 where $n$ is in a 3-cycle. We begin by noting that if a permutation avoids 132 and is comprised \emph{only} of 3-cycles, then it strongly avoids 132. This is because if $\pi$ is composed only of $3$-cycles, then $\pi^2=\pi^{-1}$ and $\pi$ avoids 132 if and only if $\pi^{-1}$ does. In \cite{AG2022}, the authors enumerate such permutations. We restate the result here (in a different, but equivalent way) in the context of this current paper.

\begin{lemma}[{\cite[Theorem~3.12]{AG2022}}] \label{lem:3cycles} The number of permutations in $\A_n$ comprised of only 3-cycles is given by the generating function
\[b(x) = \frac{2c(x^3c(x^3)) - 2}{2-c(x^3c(x^3))} \]
where $c(x)$ is the generating function for the Catalan numbers.
\end{lemma}


We will ultimately show that for $\pi\in\A_n$, if $n$ is in a 3-cycle, then $\pi$ is actually composed of an ``outer layer'' of 3-cycles surrounding another strongly 132-avoiding permutation. Since these 3-cycles can then be deleted, this makes this case very amenable to counting. 

Let us start by showing that if $n$ is a 3-cycle of the form $(n,a,b)$ with $a<b$, then all elements between $b$ and $n$ are also in 3-cycles.

\begin{lemma}\label{lem:after n}
    Suppose $n\geq 3$ and $\pi\in \A_n$. If $n$ is part of the 3-cycle $(n,a,b)$ with $a<b$, then the elements in $[b+1, n-1]$ are each in a 3-cycle as well. 
\end{lemma}

\begin{proof}
    Suppose $\pi$ contains the cycle $\pi=(n,a,b)$ with $a<b.$ Then 
    \[\pi = \pi_1\ldots \pi_{a-1} b \pi_{a+1}\ldots \pi_{b-1} n \pi_{b+1} \ldots \pi_{n-1} a.\]
    Let us first note that since $\pi$ avoids 132 and $\pi_n=a$, we must have that all elements larger than $a$ must appear before all elements smaller than $a.$ Since $\pi_b=n$, this means that $b\leq n-a$. Additionally, we have that 
    \[
    \pi([n-a+1,n-1]) = [1,a-1]. 
    \]
    If $b<n-a$, we must also have
    \[
    \pi([b+1,n-a]) = [a+1, n-b] 
    \]
    since all elements to the right of $\pi_b=n$ must be smaller than all elements to the left of $n$.
    Next, notice that 
    \[\pi([1,a-1]) = [b-a+1, b-1],\]
    that is, the first $a-1$ elements of $\pi$ consist of the largest elements that are less than $b$. If this were not so, then if $\pi([1,a-1])$ had at least one element less than $b$, $\pi_a=b$ would be the ``3'' in a 132 pattern. If each element of $\pi([1,a-1])$ were greater than $b$, then $1\pi^2_{n-1}\pi^2_n=1\pi_{\pi_{n-1}}b$ would be a 132 pattern in $\pi^2$ since $\pi_{n-1}<a$.

    Now let us consider $\pi^2$, which also avoids $132.$ We have 
    \[\pi^2=\pi^2_1\ldots \pi^2_{a-1} n \pi^2_{a+1}\ldots \pi^2_{b-1} a \pi^2_{b+1} \ldots \pi^2_{n-1} b. \]
    Since all elements that appear before $n$ must be larger than those that appear after $n$, we must have 
    \[\pi^2([1,a-1]) = [n-a+1,n-1],\]
    which implies that 
    \[
    \pi([b-a+1,b-1]) = [n-a+1,  n-1].
    \]
    Thus the sets $A = [1, a-1]$, $B = [b-a+1,  b-1]$, and $C = [n-a+1,  n-1]$ map to each other under $\pi$. That is, $\pi(A) = B, \pi(B)=C,$ and $\pi(C) = A.$ 
    Note this implies that in $\pi^2,$ the element 1 appears before the element $a$ and thus $\pi^2_b<\pi^2_{b+1}<\ldots<\pi_n^2$ since $\pi^2$ avoids 132. In particular, this means that $\pi^2_{n-a+1}\ldots\pi_{n-1}^2=(b-a)\ldots (b-1).$ Additionally, in $\pi,$ $\pi_{b-a}\ldots \pi_{b-1} = (n-a+1)\ldots (n-1)$ is increasing since otherwise, $b$ would be the 1 in a 132 pattern. Taken together, this implies that $\pi^3_{n-a+1}\ldots\pi_{n-1}^3=(n-a+1)\ldots (n-1),$ and so all elements in $A, B,$ and $C$ are in 3-cycles. 

    It remains to show that the elements in $[b+1,  n-a]$ 
    are each in a 3-cycle as well. Note that 
    \[\pi([{b+1}, {n-a}]) = [a+1,  n-b]\]
    since $\pi$ avoids 132 and that
    \[\pi^2([a+1,{n-b}]) = [b+1,  n-a]\]
    since $\pi^2$ avoids 132 and ends with $b$ (thus implying the first $n-b$ elements of $\pi^2$ are greater than $b$, and we have already determined the elements in positions 1 through $a$).
    For identical reasons as above, we can determine that $\pi^3_{a+1}\ldots\pi^3_{n-a}=(a+1)\ldots (n-a),$ and so the elements $[a+1, n-a]$ as well as $[b+1,  n-a]$ are in 3-cycles. 
\end{proof}

\begin{remark}\label{rem:inverse}
    A permutation $\pi\in\S_n$ strongly avoids 132 if and only if $\pi^{-1}$ strongly avoids 132.
\end{remark}

\begin{corollary}\label{cor:after n}
    Suppose $n\geq 3$ and $\pi\in \A_n$. If $n$ is part of the 3-cycle $(n,b,a)$ with $a<b$, then the elements in $[b+1,  n-1]$ are each in a 3-cycle as well. 
\end{corollary}
\begin{proof}
    This follows immediately from Lemma~\ref{lem:after n} and Remark~\ref{rem:inverse} by considering the inverse permutation. The inverse of a permutation can be obtained by reversing all of its cycles and so $(n,a,b)$ is a 3-cycle in $\pi^{-1}.$ By Lemma~\ref{lem:after n}, we know the elements in $[b+1, n-1]$ are all in 3-cycles in $\pi^{-1}$, and thus must be in $\pi$ as well.
\end{proof}

\begin{theorem} If $a_3(x)$ is the generating function for permutations in $\A_n$ with $n$ in a 3-cycle, then
    \[a_3(x)  = \left[2-\frac{2}{c(x^3c(x^3))}\right]a(x)\] where $a(x)$ is the generating function for permutations that strongly avoid 132. 
\end{theorem}
\begin{proof}
    Suppose $n$ is in a 3-cycle $(n,a,b)$ or $(n,b,a)$ with $a<b$.
    Let  $A = [1, a]$, $B = [b-a+1, b]$, and $C = [n-a+1, n]$. Since by Lemma~\ref{lem:after n} and Corollary~\ref{cor:after n}, these are exactly the elements of $a$ different 3-cycles, deleting all elements in each of these sets will give us a permutation that still avoids 132 and still preserves the cycle form of what remains, implying that $\pi^2$ will also still avoid 132. 

    Let us refer to the permutation that remains after deleting these elements as $\pi'\in\S_{n'}$ with $n'=n-3a$. Note that by Lemma~\ref{lem:after n}, if $a\neq n-b$, then we must have that $n'$ is also in a 3-cycle. The only way that $n'$ is not in a 3-cycle is if $a=n-b$ and so the elements of $B\cup C$ form a contiguous set. 

    Therefore, if we repeat this process until the largest element of the permutation remaining after deletion is not in a 3-cycle, we determine that the original permutation $\pi$ must have been composed from some $\beta\in\A_m$ with $\beta_m$ not in a 3-cycle and some $\alpha\in\A_{3k}$ composed only of 3-cycles by writing \[
    \pi = \alpha_1'\ldots \alpha_k' \beta_1'\ldots \beta_m' \alpha_{k+1}'\ldots \alpha_{3k}'
    \] where $\alpha'_i=\alpha_i+k+m$ if $1\leq i\leq 2k$, $\alpha'_i=\alpha_i$ if $2k+1\leq i\leq 3k$, and $\beta_i'=\beta_i+k$ for $1\leq i \leq m$.
    
    Thus we have that
    \[ a_3(x) = b(x)(a(x)-a_3(x))\]
     where $b(x)$ is as defined in Lemma~\ref{lem:3cycles} and $a(x)$ is the generating function for all permutations that strongly avoid 132,
    from which the theorem statement follows.
\end{proof}

\section{$\A_{n,k}$ with $k \geq 4$}

Consider the table in Figure~\ref{fig:ank}. For $k \geq 4$, the number of permutations that strongly avoid 132 with $n$ in a cycle of length $k$ appears to be very small and is often 0. In fact, we show that if $n$ is in a cycle of length $k$, we must have $k$ divides $n$. For the bulk of this section, we only consider those permutations $\pi$ where $n$ is in position $b>\frac{n}{2}$. The other case where $b<\frac{n}{2}$ is dealt with by taking the inverse of $\pi,$ since if $b>\frac{n}{2}$ and $\pi$ avoids 132, we must have $\pi_n<\frac{n}{2}.$

This section begins by providing more information about $\pi$ when $\pi$ contains the cycle $(a,b,n,c,\ldots)$.

\begin{lemma}\label{lem:4cycle-1} Let $n \geq 4$ and let $\pi \in \A_n$. Further suppose that $n$ is in a cycle of length at least 4 with $\pi_a=b$, $\pi_b=n$, $\pi_n=c$, and $b > \frac{n}{2}$. Then $a < b,$ $c \leq b-a$, and $a \leq 2b-n$. Furthermore,
\begin{enumerate}[label=(E{\arabic*})]
\item \label{eqn:1-ba}  $\pi([1,b-a]) = [n-b+1, n-a]$,
\item \label{eqn:ba-b} $\pi([b-a+1, b-1]) = [n-a+1, n-1]$,
\item \label{eqn:1-a}  $\pi([1, a-1]) = [b-a+1,b-1]$,
\item \label{eqn:a-b}  $\pi([a+1,b-1]) = [n-b+1, b-a] \cup [b+1,n-1]$,
\item \label{eqn:b-nc} $\pi([b+1, n-c]) = [c+1,n-b]$, and
\item \label{eqn:nc-n}  $\pi([n-c+1, n-1]) = [1,c-1]$
\end{enumerate}
\end{lemma}

\begin{proof}
We begin by denoting $\pi_c=d.$ In other words, $\pi$ contains the cycle $(a,b,n,c,d,\ldots)$. We note that $a,b,n,$ and $c$ must all be distinct but there is a possibility that $d=a$ if the cycle is length exactly 4. Because $\pi$ avoids 132 and $\pi_b = n$, we have
\begin{equation}\label{eqn:pi-front1} \pi([1,b-1]) = [n-b+1,n-1] \end{equation}
and
\begin{equation}\label{pi-end1} \pi([b+1,n]) = [1,n-b]. \end{equation}
We now show that both $a$ and $c$ are less than $b$. Since $b > \frac{n}{2},$ we have $b \geq n-b+1,$ and thus $a\in[1,b-1]$ (or $a < b$) since $\pi_a=b.$ Similarly, since $\pi_n=c,$ we have $c \leq n-b < b$ as desired. To continue, we note that $\pi([n-c+1,n-1]) = [1,c-1]$ since $\pi$ avoids 132 and $\pi_n=c$. Thus, ~\ref{eqn:nc-n} holds. This fact in conjunction with Equation~(\ref{pi-end1}) proves ~\ref{eqn:b-nc} as well.

Since $\pi_a = b$, elements in $\pi([1,a-1])$ must either be all larger than $b$ or must contain the numbers immediately preceding $b$. Since there are $a-1$ positions, we have 
\begin{equation}\label{eqn:pi1-asub}
\pi([1,a-1]) \subseteq [b-a+1,n-1].
\end{equation}
We will show that in fact $\pi([1,a-1]) = [b-a+1,b-1]$ by 
looking at $\pi^2.$ Since $\pi^2_a=n$, we have \begin{equation}\label{eqn:pi2front} \pi^2([1,a-1]) = [n-a+1, n-1] \end{equation} and $\pi^2([a+1, n]) = [1,n-a]$. We know that $c = \pi^2_b \in \pi^2([a+1,n-1])$, and thus $c \leq n-a.$ Similarly, $d = \pi^2_n$ implies that $d \leq n-a.$ By taking $\pi^{-1}$ of both sides of Equation (\ref{eqn:pi2front}) and combining this with Equation (\ref{eqn:pi1-asub}), we see that
\begin{equation} \label{eqn:weird1}
\pi([1,a-1]) = \pi^{-1}([n-a+1,n-1]) \subseteq [b-a+1,n-1]
\end{equation}
However, since $a < b,$ we have that $[n-a+1,n-1] \subseteq [n-b+1,n-1]$. Combining this with Equation (\ref{eqn:pi-front1}) gives
\begin{equation} \label{eqn:weird2}
\pi^{-1}([n-a+1,n-1]) \subseteq \pi^{-1}([n-b+1,n-1]) = [1,b-1].
\end{equation}
Combining Equations (\ref{eqn:weird1}) and (\ref{eqn:weird2}) yields
\[ \pi^{-1}([n-a+1,n-1]) \subseteq [b-a+1,n-1] \cap [1,b-1] = [b-a+1,b-1].\]
Since $[n-a+1,n-1]$ contains $a-1$ elements and $[b-a+1,b-1]$ contains $a-1$ elements, we have
\[ \pi^{-1}([n-a+1,n-1]) = [b-a+1,b-1],\]
and thus \ref{eqn:ba-b} is true.

Using \ref{eqn:ba-b} along with Equation (\ref{eqn:pi-front1}) also proves \ref{eqn:1-ba}. Furthermore, combining Equation (\ref{eqn:pi2front}) with \ref{eqn:ba-b} yields
\[ \pi^2([1,a-1]) = [n-a+1,n-1] = \pi([b-a+1,b-1]). \]
Taking $\pi^{-1}$ of both sides of this equation yields \ref{eqn:1-a} as desired. Combining Equation (\ref{eqn:pi-front1}) with \ref{eqn:1-a} proves \ref{eqn:a-b} as well.

We now show that $c \leq b-a$ and $a \leq 2b-n.$  Consider \ref{eqn:a-b}. Since $[a+1,b-1]$ contains $n-b-1$ elements and $[b+1,n-1]$ contains $b-a-1$ elements, we have $a \leq 2b-n$. 
Also, because $c < b,$ and $\pi_c=d$, we have $d \in [n-b+1,n-1].$ In particular, this implies that $n-b+1 \leq d$ and thus $c \leq n-b < d.$  We also note that $c \leq n-b = n-2b+b \leq -a +b$, and so $c \leq b-a.$
\end{proof}

Lemma \ref{lem:4cycle-1} gives information about the structure of $\pi$ when $\pi$ strongly avoids 132 and $n$ is in a cycle of length 4. Notice that \ref{eqn:1-ba} and \ref{eqn:1-a} both give information about the beginning elements of $\pi$, while \ref{eqn:ba-b} and \ref{eqn:a-b} give information about the middle of $\pi$. By comparing $b-a$ and $a$, we can refine this information. It turns out that $b-a \leq a$ occurs only in the case where $b \geq \frac{2n}{3}$, while $a < b-a$ only occurs when $\frac{n}{2} < b < \frac{2n}{3}$. The following lemma provides more information about the structure in these cases.

\begin{lemma}\label{lem:cyclelength} Let $n \geq 4$ and let $\pi \in \A_n$. Further suppose that $n$ is in a cycle of length at least 4 with $\pi_a=b$ and $\pi_b=n$, and $b > \frac{n}{2}$.
\begin{enumerate}
\item If $b \geq \frac{2n}{3}$, then $a=2b-n$.
\item If $\frac{n}{2} < b< \frac{2n}{3}$, then $n$ is in a cycle of at least length 5 with $\pi_n=n-b$ and $\pi_{n-b} = 2n-2b$.
\end{enumerate}

\end{lemma}

\begin{proof}

This proof compares $b-a$ and $a$ and examines the two cases in detail. We will show that if $b-a \leq a,$ then $b \geq \frac{2n}{3}$, and if $a< b-a$, then $\frac{n}{2} < b< \frac{2n}{3}$.

In the first case, suppose $b -a \leq a$, and consider \ref{eqn:ba-b} and \ref{eqn:a-b}. Then
\[ [n-b+1,b-a] \cup [b+1,n-1] = \pi([a+1,b-1]) \subseteq \pi([b-a+1,b-1]) = [n-a+1,n-1].  \]
Because $[n-b+1,b-a] \supset [n-a+1,n-1]$, we must have $[n-b+1,b-a] = \emptyset$ (since $a < b$), and thus $b-a < n-b+1$ or $a \geq 2b-n.$ Combining this with the result in Lemma~\ref{lem:4cycle-1} yields $a=2b-n$. Furthermore, since $a=2b-n$ and $b-a \leq a,$ we have $2n \leq 3b$ or $b \geq \frac{2n}{3}.$

We now consider the case where $a < b-a$. We begin by denoting $\pi_n=c$, $\pi_c=d$ and recall from the proof of Lemma~\ref{lem:4cycle-1} that $c < d$.  Since $a \in [1,b-a]$, by \ref{eqn:1-ba}, we have $\pi_a = b \in [n-b+1,n-a]$. Thus $b \leq n-a$, or equivalently, $a \leq n-b$. Since $a \in [1,n-b]$, by Equation (\ref{pi-end1}), we can choose $k \in [b+1,n-1]$ so that $\pi_k=a.$ In $\pi^2,$ we then have the pattern $\pi^2_b\pi^2_k\pi^2_n = cbd.$ Because $\pi^2$ avoids 132 and $c < d$ and $c < b$, we must have $d > b$. Furthermore, since $\pi_c = d > b,$ by \ref{eqn:1-a}, we must have $c > a.$

Summarizing our results so far, in the case where $a < b-a$, we have $a < c \leq b-a < b < d < n.$ 
By combining \ref{eqn:ba-b} and \ref{eqn:a-b}, we have
\begin{equation}\label{eqn:pi-aba} \pi([a+1, b-a]) = [n-b+1,b-a] \cup [b+1,n-a].
\end{equation} Since $c \in [a+1,b-a]$, and $\pi_c=d >b$, we see that $[b+1,d] \subseteq \pi([a+1,c])$ otherwise $\pi_a\pi_c=bd$ followed by a missing element from $[b+1,d]$ would be a 132-pattern in $\pi.$ Since $[b+1,d]$ contains $d-b$ elements and $[a+1,c]$ contains $c-a$ elements, we have $d-b \leq c-a$ or $d-c \leq b-a$. We also have
$\pi([a+1,c]) \subseteq [1,d]$ otherwise $\pi_a=b$ followed by something larger than $d$ followed by $\pi_c=d$ would be a 132 pattern in $\pi$. Using this information along with Equation (\ref{eqn:pi-aba}), we have
\[ [b+1,d] \subseteq \pi([a+1,c]) \subseteq [n-b+1,b-a] \cup [b+1,d].\]
Also, $\pi([a+1,c])$ must contain exactly $c-a - (d-b)$ elements from $[n-b+1,b-a].$ It must be the largest of these otherwise $\pi$ contains a 132 pattern. Thus
\begin{equation}\label{eqn:pi-ac}
\pi([a+1,c-1]) = [b+1,d-1] \cup [d-c+1,b-a].
\end{equation}
Combining Equation (\ref{eqn:pi-ac}) with Equation (\ref{eqn:pi-aba}), we also have
\begin{equation}\label{eqn:pi-cba} \pi([c+1,b-a]) = [n-b+1,d-c] \cup [d+1,n-a]. \end{equation}

We now turn our attention to $\pi^2.$ We note that $\pi^2_c = \pi_d \in [1,n-b]$ by Equation (\ref{pi-end1}) since $d > b$. Also, $\pi^2_n=d \geq n-b$. Since $\pi^2$ avoids 132, we have
\begin{equation}\label{eqn:pi2-cn} \pi^2([c+1,n-1]) \subseteq [1,d-1].\end{equation}  Combining this result with \ref{eqn:b-nc} and Equation (\ref{eqn:pi-cba}), we have \begin{align*} \pi^2([b+1,n-c]) &= \pi([c+1,n-b])\\ &\subseteq ([n-b+1,d-c] \cup [d+1,n-a]) \cap [1,d-1] = [n-b+1,d-c].\end{align*}
Now $[c+1,n-b]$ contains $n-b-c$ elements and $[n-b+1,d-c]$ contains $d-c-n+b$ elements, and thus $2n-2b \leq d.$

Similarly, using the facts that $d-c \leq b-a$, $c \leq n-b$, and Equations (\ref{eqn:pi-cba}) and (\ref{pi-end1}) along with Equation (\ref{eqn:pi2-cn}), we have
\begin{align*} \pi^2([c+1,b-a]) &= \pi([n-b+1,d-c]) \cup \pi([d+1,n-a])\\
 &\subseteq \pi([c+1,b-a]) \cup \pi([b+1,n])\\
 &\subseteq ([n-b+1,d-c] \cup [d+1,n-a] \cup [1,n-b]) \cap [1,d-1]\\
 &\subseteq [n-b+1,d-c] \cup[1,n-b].
 \end{align*}
In particular, we note that $\pi([n-b+1,d-c])$ must be contained in $[n-b+1,d-c].$ Since these sets are the same size, equality must hold and
\[ \pi([n-b+1,d-c])=[n-b+1,d-c].\]
But recall that \[ \pi([c+1,n-b]) \subseteq [n-b+1,d-c] = \pi([n-b+1,d-c]), \]
and so $[c+1,n-b] \subseteq [n-b+1,d-c]$. However, $c \leq n-b$ which implies $c=n-b$.

Since $c=n-b$, \ref{eqn:nc-n} becomes \[ \pi([b+1,n-1]) = [1,n-b-1],\] and in particular, $\pi_d < n-b-1.$
If $[n-b+1,d-n+b]$ is not the empty set, then $\pi^2_{n-b}\pi^2_{n-b+1}\pi^2_b = \pi_d (\pi^2_{n-b+1})(n-b)$ is a 132 pattern since $\pi^2_{n-b+1} \in [n-b+1,d-n+b]$. Thus, $[n-b+1,d-n+b]=\emptyset$ which implies that $n-b +1 > d-n+b$, or $d < 2n-2b+1.$ Since $2n-2b \leq d,$ we have equality and $d=2n-2b.$  Since $b < d,$ we now have $3b< 2n$ or $b < \frac{2n}{3}.$
\end{proof}

Using the results in the preceding lemmas, we have more information about the structure of $\pi$ when $n$ is in a cycle of length at least 4.  We summarize these results in the following two propositions. The first proposition describes $\pi$ when $b \geq \frac{2n}{3}$.

\begin{proposition}\label{prop:pi-form1} Let $n \geq 4$ and let $\pi \in \A_n.$ Further suppose that $n$ is in a cycle of length at least 4 with $\pi_b = n$ and $b \geq \frac{2n}{3}.$ Then $\pi=\pi_1\pi_2\cdots\pi_n$ can be written as
\[\pi_k = \begin{cases} \alpha_k^{-1} + n-b & \text{if } k \in [1,n-b]\\
k+n-b & \text{if } k \in [n-b+1, b]\\
\alpha_{k-b} & \text{if } k \in [b+1, n]
\end{cases}
\]
where $\alpha \in \S_{n-b}(132)$. Visually, we have
\[ \pi= \underline{\alpha^{-1}} (2n-2b+1)(2n-2b+2)\cdots (n)\ \alpha\]
where $\underline{\alpha^{-1}}$ is the inverse of $\alpha$ with each element shifted by $n-b$. 
\end{proposition}

\begin{proof}

If $\pi_b=n$ and $\pi_a=b$ with $b \geq \frac{2n}{3},$ then $a=2b-n$, or equivalently, $b-a = n-b$ by Lemma~\ref{lem:cyclelength}. Thus, \ref{eqn:1-ba} can be rewritten as
\[ \pi([1, n-b] = [n-b+1, 2n-2b]. \] Combining this results with Equation (\ref{eqn:pi-front1}) shows that
\[ \pi[n-b+1, b-1] = [2n-2b+1, n-1].\]

Since all of these elements are larger than $\pi_1$ and $\pi$ avoids 132, they must appear in increasing order. Visually, we have
\[ \pi = \pi_1\pi_2\cdots\pi_{n-b} (2n-2b+1)(2n-2b+2)\cdots n\  \pi_{b+1}\pi_{b+2}\cdots\pi_n\] where
\[ \pi([1,n-b]) = [n-b+1,2n-2b] \quad \text{and} \quad \pi([b+1,n])=[1,n-b].\]
Furthermore, in $\pi^2$, we have \[ \pi^2([b+1,n]) = \pi([1,n-b]) = [n-b+1, 2n-2b].\] Since 1 is not in $[n-b+1,2n-b]$, it must appear in $\pi^2$ before position $b+1$. Thus, in $\pi^2,$ all of the elements in $\pi^2([b+1,n])$ must occur in increasing order. Specifically, we have $\pi^2_{b+k} = n-b+k$ for $k \in [1, n-b]$, or $\pi_{b+k} = \pi^{-1}_{n-b+k}$ for all $k$ in $[1,n-b]$. In other words, the values of $\pi_1\pi_2\cdots\pi_{n-b}$ are forced by the values of $\pi_{b+1}\pi_{b+2}\cdots\pi_{n}$.  Also, because $\pi$ avoids 132, we must have the subword $\pi_{b+1}\pi_{b+2}\cdots\pi_n$ avoids 132, and thus is contained in $\S_{n-b}(132).$
\end{proof}

The next proposition gives the structure of $\pi$ when $\frac{n}{2} < b < \frac{2n}{3}$.

\begin{proposition}\label{prop:pi-form2} Let $n \geq 4$ and let $\pi \in \A_n$. Further suppose that $n$ is in a cycle of length at least 4 with $\pi_b = n$ and $\frac{n}{2} < b < \frac{2n}{3}.$ Then $\pi = \pi_1\pi_2\cdots\pi_n$ can be written as
\[ \pi_k = \begin{cases} \alpha_k^{-1} + n-b & \text{if } k \in [1,2b-n]\\
k + n-b & \text{if } k \in [2b-n+1, b]\\
\alpha_{k-b} & \text{if } k \in [b+1,3b-n]\\
k-b & \text{if } k \in [3b-n+1, n] \end{cases}\]
where $\alpha \in \S_{2b-n}(132).$ Visually, we have
\[ \pi = \underline{\alpha^{-1}}\ (b+1)(b+2)\cdots (n)\ \alpha \ (2b-n+1)(2b-n+1)\cdots(n-b)\]
where $\underline{\alpha^{-1}}$ is the inverse of $\alpha$ with each element shifted by $n-b$.
\end{proposition}

\begin{proof}

Since $b < \frac{2n}{3}$, we must have $a < b-a$ and that $\pi$ contains the cycle $(a, b, n, n-b, 2n-2b, \cdots)$ by Lemma~\ref{lem:cyclelength}. Consider again $\pi^2$ while noting that $c = n-b$ and $d=2n-2b$. By Equation~(\ref{eqn:pi-ac}), \ref{eqn:nc-n}, and Equation~(\ref{eqn:pi-cba}) we have
\begin{align*}
\pi^2([a+1, n-b]) &= \pi([b+1, 2n-2b]) \cup \pi([n-b+1,b-a])\\
&= \pi([b+1, 2n-2b]) \cup [2n-2b+1, n-a]\\
&\subseteq [1,n-b-1] \cup [2n-2b+1, n-a].
\end{align*}

Because $b > n-b$ and $\pi^2_b = n-b$, we must have all of the elements in $[2n-2b+1,n-a]$ appear before all of the elements in $\pi([b+1, 2n-2b]) \subseteq [1,n-b-1]$ in $\pi^2$ in order for $\pi^2$ to avoid 132. Thus,
\begin{equation*}
    \pi^2([a+1, 2b-n]) = [2n-2b+1, n-1]
\end{equation*}
which implies that
\begin{equation}\label{eqn:pi-a2bn}
    \pi([a+1, 2b-n]) = \pi^{-1}([2n-2b+1, n-1]) = [n-b+1, b-a].
\end{equation}
Combining this with \ref{eqn:1-a} yields
\begin{equation}\label{eqn:pi-frontcase2}
\pi([1,2b-n]) = [n-b+1, b]
\end{equation}
Furthermore, combining Equation (\ref{eqn:pi-a2bn}) with Equation (\ref{eqn:pi-ac}) gives
\begin{equation}
\pi([2b-n+1, n-b-1]) = [b+1, 2n-2b+1],
\end{equation}
and because these elements are all larger than $b$ which occurs in a position preceding these, they must all be increasing. Combining Equations (\ref{eqn:pi-cba}) with \ref{eqn:ba-b} gives
\[ \pi([n-b, b]) = [2n-2b, n],\]
and again, these elements must be increasing. Thus we have $\pi_k = k + n-b$ for $k \in [2b-n+1,b]$ as desired.

Notice that $\pi^2_b = n-b$ and 
\[ \pi^2([b+1, n]) = \pi([1,n-b]) = [n-b+1, 2n-2b]. \]
Since these elements are all larger than $n-b$ and $\pi^2$ avoids 132, they must be increasing and thus $\pi^2_k = n-2b+k$ for $k \in [b+1,n]$. In particular, suppose $k \in [3b-n+1, n]$. We then have \[ \pi(k) = \pi^{-1}(n-2b+k) = \pi^{-1}((k-b) + n-b).\] Since $k-b \in [2b-n+1, n-b]$, we have $\pi_k = k-b$. Thus $\pi_k = k-b$ for $k \in [3b-n+1, n]$ as desired. Furthermore, for $k \in [b+1, 3b-n]$, $\pi_k = \pi^{-1}_{n-2b+k} \in [1,2b-n]$. We again see that the elements in $\pi([1, 2b-n])$ are forced by those in $\pi([b+1,3b-n])$ as desired.
\end{proof}

Propositions \ref{prop:pi-form1} and \ref{prop:pi-form2} show what permutations must look like if $n$ is in a cycle of length at least 4. We can also consider the converses of these propositions. For Proposition \ref{prop:pi-form1}, the converse question is: If a permutation $\pi$ can be written as
\[ \pi = (\alpha^{-1})^{*} (2n-2b+1)(2n-2b+2)\cdots (n)\ \alpha\] where $\alpha$ is 132-avoiding and $b \geq \frac{2n}{3},$ what can we say about the length of the cycle containing $n$?  We can consider the converse of Proposition \ref{prop:pi-form2} as well. The following example illustrates the cycle structure of $\pi$ when $n=12$ and $\pi$ has the structure provided in Propositions \ref{prop:pi-form1} and \ref{prop:pi-form2}.



\begin{example} Let $n=12$. Suppose $b=7 < \frac{2n}{3}$. Then there are two choices for $\alpha \in \S_{2b-n}(132) = \S_{2}(132)$. In both cases, the resulting permutation is cyclic:
\begin{align*}
\fbox{6\ 7}\ 8\ 9\ 10\ 11\ 12\ \fbox{1\ 2}\ 3\ 4\ 5 &= (1, 6, 11, 4, 9, 2, 7, 12, 5, 10, 3, 8)\\
\fbox{7\ 6}\ 8\ 9\ 10\ 11\ 12\ \fbox{2\ 1}\ 3\ 4\ 5 &= (1, 7, 12, 5, 10, 3, 8, 2, 6, 11, 4, 9)
\end{align*}
Now suppose $b=8 \geq \frac{2n}{3}$, $\alpha = \alpha_1\alpha_2\alpha_3\alpha_4 \in \S_4(132)$, and let 
\[ \pi = (\alpha_1^{-1} + 4)(\alpha_2^{-1}+4)(\alpha_3^{-1} + 4)(\alpha_4^{-1} + 4)\ 9\ 10\ 11\ 12\ \alpha_1\alpha_2\alpha_3\alpha_4.\]
Note that for $i \in [1,4],$ $\pi(\alpha_i) = \alpha_{\alpha_i}^{-1} + 4 = i+4.$ Thus the cycle form of $\pi$ can be written as
\[ \pi = (5, 9, \alpha_1)(6, 10, \alpha_2)(7, 11, \alpha_3)(8, 12, \alpha_4)\]
for all fourteen of these permutations. It is interesting to note that although $\pi$ has the structure provided in Proposition \ref{prop:pi-form1}, it is not the case that $n$ is in a cycle of length at least 4 so the converse of the proposition does not hold.

If $b = 9$ and $\alpha \in \S_3(132),$ then there are five possible choices for $\alpha$ and thus five permutations:
\begin{align*}
\fbox{4\ 5\ 6}\ 7\ 8\ 9\ 10\ 11\ 12\ \fbox{1\ 2\ 3} &= (1,4,7,10)(2,5,8,11)(3,6,9,12)\\
\fbox{5\ 4\ 6}\ 7\ 8\ 9\ 10\ 11\ 12\ \fbox{2\ 1\ 3} &= (1,5,8,11)(2,4,7,11)(3,6,9,12)\\
\fbox{5\ 6\ 4}\ 7\ 8\ 9\ 10\ 11\ 12\ \fbox{3\ 1\ 2} &= (1,5,8,11)(2,6,9,12)(3,4,7,10)\\
\fbox{6\ 4\ 5}\ 7\ 8\ 9\ 10\ 11\ 12\ \fbox{2\ 3\ 1} &= (1,6,9,12)(2,4,7,10)(3,5,8,11)\\
\fbox{6\ 5\ 4}\ 7\ 8\ 9\ 10\ 11\ 12\ \fbox{3\ 2\ 1} &= (1,6,9,12)(2,5,8,11)(3,4,7,10)
\end{align*}
Finally, if $b=10$ there are two possible permutations, and there is one possible permutation for $b=11:$
\begin{align*}
\fbox{3\ 4}\ 5\ 6\ 7\ 8\ 9\ 10\ 11\ 12\ \fbox{1\ 2} &=(1,3,5,7,9,11)(2,4,6,8,10,12)\\
\fbox{4\ 3}\ 5\ 6\ 7\ 8\ 9\ 10\ 11\ 12\ \fbox{2\ 1} &=(1,4,6,8,10,12)(2,3,5,7,9,11)\\
\fbox{2}\ 3\ 4\ 5\ 6\ 7\ 8\ 9\ 10\ 11\ 12\ \fbox{1} &= (1,2,3,4,5,6,7,8,9,10,11,12)
\end{align*}
\end{example}

The lengths of the cycles in the above example appear to be related to $n$ and the $\gcd(n,b)$. In fact, we can find the exact lengths of the cycle containing $n$ for all $n$ and $b$ as shown in the following propositions.

\begin{proposition}\label{prop:cyclen1} Let $n \geq 4,$ $b \geq \frac{2n}{3}$, $\alpha \in \S_{n-b}(132)$, and suppose $\pi = \pi_1\pi_2\cdots\pi_n$ with
\[\pi_k = \begin{cases} \alpha_k^{-1} + n-b & \text{if } k \in [1,n-b]\\
k+n-b & \text{if } k \in [n-b+1, b]\\
\alpha_{k-b} & \text{if } k \in [b+1, n].
\end{cases}
\] Then $\pi \in \A_n$ and $n$ is in a cycle of length $\dfrac{n}{\gcd(n, b)}$.
\end{proposition}

\begin{proof} To show that $\pi \in \A_n,$ we need to show that $\pi^2$ avoids 132. We do this by writing the structure of $\pi^2$ explicitly. If $k \in [1,n-b]$, we see that $\alpha^{-1}_k + n-b \in [n-b+1, 2n-2b] \subseteq [2n-2b,b]$, with the last subset derived from the fact that $2n \leq 3b.$ Thus for $k \in [1,n-b]$, we have
\[ \pi^2_k = \pi(\alpha_k^{-1} + n-b) = (\alpha_k^{-1} + n-b) + n-b = \alpha_k^{-1} + 2n-2b.\]
For values of $\pi^2$ in the set $[n-b+1,b]$, we partition this set into two other sets. For $k \in [n-b+1,2b-n]$, we note that $k+n-b \in [2n-2b+1,b].$ Thus for $k \in [n-b+1,2b-n]$,
\[ \pi^2_k = \pi(k+n-b) = (k+n-b) + n-b = k + 2n-2b.\]
For the second set in the partition, that is, for $k\in [2b-n+1,b]$, we have
\[ \pi^2_k = \pi(k+n-b) = \alpha_{k+n-2b}.\]
Finally, for $k \in [b+1,n],$
\[ \pi^2_k = \pi(\alpha_{k-b}) = \alpha^{-1}(\alpha_{k-b}) + n-b = k + n-2b.\]

Putting these results together shows us that $\pi^2 = \pi^2_1\pi^2_2\cdots\pi^2_n$ is given by
\[ \pi^2_k = \begin{cases} \alpha_k^{-1} + 2n-2b & \text{if } k \in [1,n-b]\\
k + 2n-2b & \text{if } k \in [n-b+1, 2b-n]\\
\alpha(k+n-2b) & \text{if } k \in [2b-n+1,b]\\
k+n-2b & \text{if } k \in [b+1,n]
\end{cases}
\]
which is 132-avoiding as desired.

We now consider the cycle structure of $\pi.$ Let $n=q(n-b)+r$ with $r< n-b$. We know that $q\geq3$ since $3b \geq 2n$ implies $n \geq 3(n-b)$. Notice that for $k \in [n-b+1, b]$, we have $\pi_k=k+(n-b)$, and iteratively applying $\pi$ to these elements just adds $n-b$ each time with a result that is less than $n$.

Now for $k \in [b+1, n]$, we iteratively apply $\pi$ twice. Here we have $\pi^2_k = k + 2(n-b) - n.$ In this case, we applied $\pi$ twice, which added $(n-b)$ twice and then subtracted $n$. Equivalently, we can think of this as adding $2(n-b)$ and then taking the result modulo $n$. Thus, if we start with $n$ and apply $\pi$ iteratively, we will return to $n$ in $m$ steps exactly when $m(n-b)\equiv0\bmod n.$ The smallest such $m$ is exactly $n/\gcd(n,n-b)$ as desired. \end{proof}

\begin{proposition}\label{prop:cyclen2} Let $n \geq 4,$ $\frac{n}{2} < b < \frac{2n}{3}$, $\alpha \in \S_{2b-n}(132)$, and suppose $\pi = \pi_1\pi_2\cdots\pi_n$ with
\[ \pi_k = \begin{cases} \alpha_k^{-1} + n-b & \text{if } k \in [1,2b-n]\\
k + n-b & \text{if } k \in [2b-n+1, b]\\
\alpha_{k-b} & \text{if } k \in [b+1,3b-n]\\
k-b & \text{if } k \in [3b-n+1, n] \end{cases}\]
Then $\pi \in \A_n$ and $n$ is in a cycle of length $\dfrac{n}{\gcd(n, b)}$.
\end{proposition}

\begin{proof} We begin by computing $\pi^2 = \pi_1^2\pi_2^2\cdots \pi_n^2$ in order to show that $\pi \in \A_n$. We examine five cases based on partitioning $[1,n]$ into five sets. First suppose $k \in [1,2b-n]$ which means $\alpha_k^{-1} +n-b \in [n-b+1,b].$ In this case
\[ \pi^2_k = \pi(\alpha_k^{-1}+n-b) = \alpha_k^{-1} + 2n-2b.\]
If $k \in [2b-n+1,4b-2n]$, then $k+n-b \in [b+1, 3b-n]$ so
\[ \pi^2_{k} = \pi(k+n-b) = \alpha(k+n-2b).\]
In the case where $k \in [4b-2n+1,b],$ we have $k + n-b \in [3b-n+1,n]$ and so
\[ \pi^2_k = \pi(k+n-b) = k+n-2b.\]
For the penultimate case where $k \in [b+1,3b-n],$ we see that $\alpha_{k-b} \in [1,2b-n]$ and so
\[ \pi^2_k = \pi(\alpha_{k-b}) = \alpha^{-1}(\alpha_{k-b}) + n-b = k+n-2b.\]
Finally, if $k \in [3b-n+1,n],$ then $k-b \in [2b-n+1,n-b]$ so
\[ \pi^2_k = \pi(k-b) = k+n-2b.\]

Notice the last three cases all yield $\pi^2_k = k+n-2b$, so the piecewise function for $\pi^2$ can be condensed into just three cases:
\[
\pi^2_k = \begin{cases}
\alpha_k^{-1} + 2n-2b & \text{if } k \in [1,2b-n]\\
\alpha(k+n-2b) & \text{if } k \in [2b-n+1, 4b-2n]\\
k+n-2b & \text{if } k \in [4b-2n+1,n].
\end{cases}
\]
We see then that $\pi^2$ is 132-avoiding as desired.

Again, we consider the cycle structure of $\pi$ by iteratively applying $\pi$ to elements. If $k \in [2b-n+1],$ then $\pi_k = k + (n-b)$, and so iteratively applying $\pi$ adds $n-b$ each time with a result that is less than $n$. 

If $k \in [b+1,3b-n],$ then $k \in [4b-2n+1,n]$ since $3b<2n$ implies $4b-2n+1 < b+1.$ In this case, we iteratively apply $\pi$ twice and see that $\pi^2_k = k + 2(n-b) - n.$ In this case, we applied $\pi$ twice, which added $(n-b)$ twice and then subtracted $n$. Similar to the proof of Proposition~\ref{prop:cyclen1}, we see that if we start with $n$ and apply $\pi$ iteratively, we will return to $n$ in $m$ steps exactly when $m(n-b)\equiv0\bmod n$. The smallest such $m$ is exactly $n/\gcd(n,n-b)$ as desired.
\end{proof}

Recall that we have only dealt with the case where $b>\frac{n}{2}$ but that the case where $b<\frac{n}{2}$ is obtained by taking the inverse of $\pi.$ We do not consider the case when $b = \frac{n}{2}$ since this would imply $n$ is in a $2$-cycle.

\begin{theorem}
    Let $n\geq 4$ and $a_{n}^{\geq 4} = \sum_{k=4}^n a_{n,k}$. Then \[
    a_{n}^{\geq 4} = 2\left[\sum_{i=1}^{\lfloor (n-1)/3\rfloor} c_{i} + \sum_{i=\lceil (n+1)/3\rceil}^{\lfloor  (n-1)/2\rfloor} c_{n-2i}\right]
    \]
    where $c_i$ is the $i$-th Catalan number.
    The generating function for $a_n^{\geq 4}$ is \[a_{\geq 4}(x) = 2\left(c(x^3)-1\right)\left[\frac{x}{1-x} + \frac{x^2}{1-x^2}\right].\]
\end{theorem}

\begin{proof}
    Suppose $\pi\in\A_n$ has $n$ in a $k$-cycle with $k\geq 4$ and $\pi_b=n.$ Note that we cannot have $b=n/2$ since in that case, we would have the $b-1$ elements before $n$ smaller than the $b$ elements after $n$. In $\pi^2$, we would thus have the first $b$ elements smaller than the second $b$ elements. Since $\pi^2$ avoids 132, this would imply that $\pi^2_n=n$, and thus $n$ is in a 2-cycle. If $b<\frac{n}{2},$ the we must have $\pi_n>\frac{n}{2}$ since otherwise $1n\pi_{\pi_n}$ would be a 132 pattern in $\pi^2$. By inverses, it is enough to consider $b>\frac{n}{2}.$
    
    Propositions~\ref{prop:pi-form1} and \ref{prop:pi-form2} imply such permutations must be of the form \[ \pi= \underline{\alpha^{-1}} (2n-2b+1)(2n-2b+2)\cdots (n)\ \alpha\] where $\alpha\in\S_{n-b}(132)$ or \[ \pi = \underline{\alpha^{-1}}\ (b+1)(b+2)\cdots (n)\ \alpha \ (2b-n+1)(2b-n+1)\cdots(n-b)\]
where $\alpha\in\S_{2b-n}(132)$, and in both cases $\underline{\alpha^{-1}}$ is the inverse of $\alpha$ with each element shifted by $n-b$.  Propositions~\ref{prop:cyclen1} and \ref{prop:cyclen2} state that these are all valid permutations in $\A_n$. Provided that $n/\gcd(n,b)\geq 4$, they will have $n$ in a $k$-cycle with $n\geq 4.$ Removing the possibility that $b$ is $n/2$ or $n/3$, and using the fact that  $132$-avoiding permutations are enumerated by the Catalan numbers, the result follows. 
\end{proof}

\begin{proof}[Proof of Theorem~\ref{thm:main}]
    Considering the size of the cycle $n$ appears in and letting $a(x) = \Sav_{132}(x),$ we have
    \begin{align*}a(x) &=1+ a_1(x) + a_2(x)+a_3(x)+a_{\geq 4}(x)\\
    & = 1+  xa(x) + \frac{x^{2} c(x^{2})}{1 - x - x^{2} c(x^{2})}+ \left[2-\frac{2}{c(x^3c(x^3))}\right]a(x)+ 2\left(c(x^3)-1\right)\left[\frac{x}{1-x} + \frac{x^2}{1-x^2}\right].
    \end{align*}
    Solving for $a(x)$ gives us the result in Theorem~\ref{thm:main}.
\end{proof}

\begin{corollary}
   Let $n \geq 3$. Then
\[ a_{n,n} =  \sum_{\substack{1 \leq r <  \frac{n}{2} ,\\ \gcd(r, n) = 1}} 2c_{\min\{r, n-2r\}}
\]
where $c_n$ denotes the $n$th Catalan number. 
\end{corollary} 

\begin{proof}
Note that by Propositions~\ref{prop:pi-form1} and \ref{prop:pi-form2}, we have $n$ in an $n$-cycle exactly when $\gcd(b,n)=1.$ There are $2c_r$ permutations (including inverses) obtained from Proposition~\ref{prop:pi-form1} for each $r\in [1,\lfloor (n-1)/3\rfloor]$ with $\gcd(r,n)=1$ and $2c_{n-2r}$ permutations (including inverses) obtained from Proposition~\ref{prop:pi-form1} for each $r\in [\lceil (n+1)/3\rceil,\lfloor (n-1)/2\rfloor]$ with $\gcd(r,n)=1$. Since you cannot have $r=(n+1)/3$ for $n\geq 3$ if $\gcd(n,r)=1$, the result follows. 
\end{proof}


\section{Analytic properties of $\Sav_{132}(x)$}
In this section, we examine the analytic properties of $\Sav_{132}(x)$ in order to bound the coefficients $a_n$, that is the number of permutations of size $n$ that strongly avoid 132. In doing so, we answer Question~\ref{ques:bona} in the affirmative.  We begin by analyzing the singularities of $\Sav_{132}(n).$ It is easy to see that the generating function is analytic on $|x|<1/2$ and that it has a singularity at $x=1/2$ as seen in the following lemma.

\begin{lemma} The function $\Sav_{132}(x)$ has a singularity at $x=1/2$ and is analytic on $|x| < 1/2.$
\end{lemma}
\begin{proof}
Recall 
    \[
\Sav_{132}(x)
=\frac{
c\bigl(x^3 c(x^3)\bigr)}{{
2-(1+x)c\bigl(x^3 c(x^3)\bigr)}}\Bigg(\dfrac{1-x}{1-x-x^2 c(x^2)}
\;+\;\dfrac{2x(1+2x)\bigl(c(x^3)-1\bigr)}{1-x^2}\Bigg),
\]
where $c(x) = \frac{1-\sqrt{1-4x}}{2x}$. First note that 
\[\dfrac{1-x}{1-x-x^2 c(x^2)} = \frac{2(1-x)}{1 - 2x + \sqrt{1 - 4x^2}},\]
which only has singularities $|x| = 1/2.$
Next, notice that 
\[\dfrac{2x(1+2x)\bigl(c(x^3)-1\bigr)}{1-x^2}\]
only has singularities $|x| = 1$ and the singularity of $c(x^3)$ which is $\frac{1}{4^{1/3}} > 1/2$.
Finally, consider
\[\frac{
c\bigl(x^3 c(x^3)\bigr)}{{
2-(1+x)c\bigl(x^3 c(x^3)\bigr)}}.\]
The smallest singularity of $c(x^3c(x^3))$ occurs when $x^3c(x^3) = 1/4,$ or equivalently, when $x=(3/16)^{1/3} > 1/2.$ When considering the denominator, we see that for $|x| < 1/2,$ we have
\[ |x^3c(x^3)| \leq \dfrac{1}{8}\left|c\left(\dfrac{1}{8}\right)\right| = \frac{1-\sqrt{1/2}}{2}. \]
Thus,
\[ |1+x|\left|c\bigl(x^3 c(x^3)\bigr)\right| \leq \frac{3}{2}\cdot c\left(\frac{1-\sqrt{1/2}}{2}\right) = \frac{3}{2}\left(\frac{1-\sqrt{1-2(1-\sqrt{1/2})}}{1-\sqrt{1/2}}\right) < 2,\]
and so the denominator does not vanish on $|x| < 1/2.$
\end{proof}

In order to find the asymptotic behavior of the terms, we expand the function around 1/2.

\begin{lemma}\label{lem:epsilon} Let 
\[ f(x) = \frac{1-x}{1-x-x^2c(x^2)}.\]
Then expanding $f(x)$ around $x=1/2$ gives 
\[ f(x) \approx \frac{1}{\sqrt{2}}(1-2x)^{-1/2} \]
\end{lemma}

\begin{proof}
Let $x=1/2-\e$ as $\e \to 0^+.$ Then
\begin{align*}
    x^2c(x^2) = \frac{1-\sqrt{1-4(\frac{1}{2}-\e)^2}}{2} = \frac{1-\sqrt{4\e-4\e^2}}{2} = \frac{1}{2} - \sqrt{\e}\sqrt{1-\e}.
\end{align*}

Now, $1-x-x^2c(x^2)$ becomes
\begin{align*}1-x-x^2c(x^2) &= 1-\left(\frac{1}{2}-\e\right) - \left(\frac{1}{2} -\sqrt{\e}\sqrt{1-\e}\right)\\
&= \sqrt{\e}(\sqrt{\e} + \sqrt{1-\e}).\end{align*}
which approaches $\sqrt{\e}$ or $\sqrt{1/2-x}$ as $\e \to 0^+$. Thus, expanding $f(x)$ around $x=1/2$ gives
\[ f(x) \approx \frac{1-\frac{1}{2}}{\sqrt{\frac{1}{2}-x}} = \frac{1}{\sqrt{2}}(1-2x)^{-1/2}.\]
\end{proof}

We now note that the remaining terms in $\Sav_{132}(x)$ can be evaluated at $x=1/2$, and thus we can also expand $\Sav_{132}(x)$ around $x=1/2.$
\begin{proposition} Let $a_n$ be the number of 132 strongly avoiding permutations. Then
\[ a_n \sim K\frac{2^n}{\sqrt{n}},\]
where 
\[ K = \frac{1}{\sqrt{2\pi}} \cdot \frac{c\left(\frac{1}{8}c(\frac{1}{8})\right)}{2-\frac{3}{2}c\left(\frac{1}{8}c(\frac{1}{8})\right)} \approx 2.77826\]
\end{proposition}
\begin{proof} We begin by expanding $\Sav_{132}(x)$ around $x=1/2$. Using the expansion of $f(x)$ from Lemma~\ref{lem:epsilon}, we have
\begin{align*} \Sav_{132}(x) &\approx K_1\left(\frac{1}{\sqrt{2}}(1-2x)^{-1/2} + K_2 \right)
\end{align*}
where \[K_1 = \frac{c(\frac{1}{8}c(\frac{1}{8}))}{2 - (\frac{3}{2})c(\frac{1}{8}c(\frac{1}{8}))}\] and 
\[ K_2 = \frac{2(c(\frac{1}{8}))}{\frac{3}{4}}.\]
By using the standard asymptotic formula for coefficients of $x^n$ of $(1-x)^{-\alpha}$, which are given by
$\frac{n^{\alpha-1}}{\Gamma(\alpha)}$, we achieve the desired result.
\end{proof}

\begin{corollary}
    The growth rate of $a_n$ is 2, and furthermore, $a_n<2^n$ for $n$ sufficiently large n.
\end{corollary}




\bibliographystyle{amsplain}

\subsection*{Disclaimer}
The views expressed in this article do not necessarily represent 
the views or opinions of the U.S. Naval Academy, Department of the Navy, or Department of Defense or any of its components.

\end{document}